\documentclass[11pt,reqno]{amsart}
\usepackage{amssymb,amsthm,amsmath}
\usepackage{tikz}
\usepackage[margin=2.5cm]{geometry}
\usepackage{url}

\newtheorem{lemma}{Lemma}[section]
\newtheorem{theorem}[lemma]{Theorem}
\newtheorem{proposition}[lemma]{Proposition}

\newtheorem{conjecture}[lemma]{Conjecture}

\theoremstyle{definition}

\newtheorem{question}[lemma]{Question}

\newcommand{\reals}{\mathbb{R}}

\newcommand{\gde}{\ensuremath{G-e}}
\newcommand{\gce}{\ensuremath{G/e}}

\newcommand{\ser}{\mathbin{\oplus_S}}
\newcommand{\pll}{\mathbin{\oplus_P}}

\title{The Merino--Welsh Conjecture holds for Series--Parallel Graphs}
\author{Steven D. Noble}
\address[Steven D. Noble]{Mathematical Sciences \\ John Crank 106\\ Brunel University \\ Uxbridge UB8 3PH \\ United Kingdom}
\email{Steven.Noble@brunel.ac.uk}

\author{Gordon F. Royle}
\address[Gordon F. Royle]{School of Mathematics \& Statistics \\ University of Western Australia \\ 35 Stirling Highway\\ Nedlands 6009\\Australia}
\email{Gordon.Royle@uwa.edu.au}
\begin{document}

\begin{abstract}
The Merino-Welsh conjecture asserts that the number of spanning trees of a graph is no greater than the maximum
of the numbers of totally cyclic orientations and acyclic orientations of that graph. We prove this conjecture for the
class of series-parallel graphs.
\end{abstract}

\maketitle

\section{Introduction}

In this paper, we are concerned with the relationship between three graph parameters, namely the number of {spanning
trees}, the number of acyclic orientations, and the number of {totally cyclic orientations} of a graph. In principle, our graphs may have loops, bridges and multiple edges, although only the last will play any eventual role.  Recapping some basic terminology, recall that a  {\em spanning tree} of a graph $G$ is a set of edges inducing a connected spanning subgraph of $G$, but containing no cycle of $G$; we use $\tau(G)$ to
denote the number of spanning trees of $G$. An {\em orientation} of a graph is an assignment of a direction to 
each edge of the graph --- the orientation is called {\em acyclic} if the resulting directed graph contains no directed
cycles, and it is called {\em totally cyclic} if {\em every} edge is contained in a directed cycle. We use 
$\alpha(G)$ to denote the number of acyclic orientations of $G$, and $\alpha^*(G)$ to denote the number
of totally cyclic orientations of $G$. It is immediate that if $G$ contains a loop, then $\alpha(G) = 0$ and that if $G$ contains a bridge, then $\alpha^*(G) = 0$. 

Examination of empirical evidence led Merino and Welsh to make the following conjecture relating these
three graphical parameters:
\begin{conjecture}[Merino-Welsh Conjecture \cite{MR1772357}]\label{conj:CMregular}
For any bridgeless, loopless graph $G$,
\begin{equation}\label{eq:CMregular}
\max \{\alpha(G),\alpha^*(G)\} \geq \tau(G).
\end{equation}
\end{conjecture}

The $\max(\cdot)$ function is awkward to carry through any sort of inductive proof, but there are two very natural
generalisations of this conjecture involving the arithmetic and geometric means of $\alpha$ and $\alpha^*$, which first
appeared in Conde \& Merino \cite{MR2555382}.

\begin{conjecture}[Additive Merino-Welsh Conjecture] \label{conj:CMadd}
For any bridgeless, loopless graph $G$,
\begin{equation}\label{eq:CMadd}
\alpha(G) + \alpha^*(G) \geq 2\tau(G).
\end{equation}
\end{conjecture}

\begin{conjecture}[Multiplicative Merino-Welsh Conjecture]\label{conj:CMmult}
For any bridgeless, loopless graph $G$,
\begin{equation}\label{eq:CMmult}
\alpha(G) \alpha^*(G) \geq \tau(G)^2. 
\end{equation}
\end{conjecture}

The {\em digon}, which is the 2-vertex graph with a double edge connecting its two vertices has 2 spanning trees, 2
acyclic orientations and 2 totally cyclic orientations and thus achieves equality in all three variants of the conjecture.

Intuitively, for any fixed number of vertices, the (connected) graphs with few edges tend to be ``tree-like'' and have more acyclic orientations than spanning trees, whereas for graphs with many edges, the number of totally cyclic orientations tends to dominate the number of spanning trees. This qualitative statement was made more precise by Thomassen.

\begin{theorem}[Thomassen \cite{MR2732508}]
If $G$ is a simple graph on $n$ vertices with $m \leq 16n/15$ edges, then \[\alpha(G) > \tau(G),\] and if $G$ is a
bridgeless graph on $n$ vertices with $m \geq 4n-4$ edges then \[\alpha^*(G) > \tau(G).\]
\end{theorem}

This result raises the possibility that the Merino-Welsh conjecture might be resolved by extending upwards the value of 
$m$ for which acyclic orientations are guaranteed to dominate spanning trees, and 
downwards the value of $m$ for which totally cyclic orientations are guaranteed to dominate until the two bounds meet
and therefore cover every possibility. More precisely, Thomassen asked the following question:

\begin{question}[Thomassen \cite{MR2732508}]
If $G$ is a bridgeless, loopless graph with $n$ vertices and $m$ edges, then is it true that
\begin{align*}
\alpha(G) & \geq \tau(G) \text{ if } m \leq 2n-2, \text{ and }\\
\alpha^*(G) & \geq \tau(G) \text{ if } m \geq 2n-2?
\end{align*}
\end{question}

However, the answer to Thomassen's question is ``No''; if $G$ is the graph of Figure~\ref{thom1} consisting of $n-2$ digons and two edges arranged in a cycle (in arbitrary order), then 
\begin{align*}
\tau(G) &= 2^{n-1}+(n-2) 2^{n-3},\\
\alpha(G) &=2^n - 2,\\
\alpha^*(G) &= 2 \cdot 3^{n-2}.
\end{align*}
Therefore for $n \geq 6$, it follows that $\alpha(G) < \tau(G)$. The dual graphs $H  = G^*$ provides examples of
$n$-vertex graphs with $2n-2$ edges where $\alpha^*(H) < \tau(H)$.  However, we know of no $n$-vertex graphs with {\em strictly fewer} than $2n-2$ edges for which $\alpha(G) < \tau(G)$ and, similarly, none with {\em strictly more} than $2n-2$ edges for which $\alpha^*(G) < \tau(G)$.

\begin{figure}
\begin{center}
\begin{tikzpicture}[out=45,in=135,relative]
\tikzstyle{vertex}=[circle, fill=gray,draw = black, inner sep = 0.6mm]
\node [vertex] (v0) at (0:1.5cm) {};
\node [vertex] (v1) at (30:1.5cm) {};
\node [vertex] (v2) at (60:1.5cm) {};
\node [vertex] (v4) at (120:1.5cm) {};
\node [vertex] (v5) at (150:1.5cm) {};
\node [vertex] (v6) at (180:1.5cm) {};
\node [vertex] (v7) at (210:1.5cm) {};
\node [vertex] (v8) at (240:1.5cm) {};
\node [vertex] (v9) at (270:1.5cm) {};
\node [vertex] (v10) at (300:1.5cm) {};
\node [vertex] (v11) at (330:1.5cm) {};
\draw  [bend right] (v4) to (v5);
\draw  [bend left] (v4) to (v5);
\draw  [bend right] (v5) to (v6);
\draw  [bend left] (v5) to (v6);
\draw  [bend right] (v6) to (v7);
\draw  [bend left] (v6) to (v7);
\draw  [bend right] (v7) to (v8);
\draw  [bend left] (v7) to (v8);
\draw [thick] (v8)--(v9);
\draw [thick] (v9)--(v10);
\draw  [bend right] (v10) to (v11);
\draw  [bend left] (v10) to (v11);
\draw  [bend right] (v11) to (v0);
\draw  [bend left] (v11) to (v0);
\draw  [bend right] (v0) to (v1);
\draw  [bend left] (v0) to (v1);
\draw  [bend right] (v1) to (v2);
\draw  [bend left] (v1) to (v2);
\draw [dashed, thick, bend right] (v2) to (v4);
\end{tikzpicture}
\caption{Graph answering Thomassen's question in the negative}
\label{thom1}
\end{center}
\end{figure}
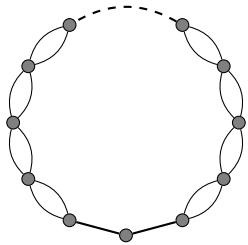

Therefore, at least for graphs with $m=2n-2$ edges, the proof must either consider the two parameters simultaneously, or relate them in a more precise way to the graph structure. In the remainder of this paper, we give a proof that the Merino-Welsh conjecture holds for {\em series-parallel graphs}, which are the graphs that arise from a single edge by repeatedly duplicating or subdividing edges in any fashion.  

Although series-parallel graphs form only a tiny subset of all graphs, the class is closed under duality and includes
numerous graphs with $m=2n-2$ edges (including the graphs described above).


\section{Deletion and Contraction}

The {\em Tutte polynomial} of a graph $G$ is the $2$-variable polynomial $T_G(x,y)$ that is equal to $1$ if the graph has no edges, and is otherwise defined recursively by 
\[
T_G(x,y) = 
\begin{cases}
x\,T_{\gce} (x,y), & \text{if $e$ is a bridge;}\\
y\,T_{\gde}(x,y), & \text{if $e$ is a loop;}\\
T_{\gde}(x,y) + T_{\gce}(x,y), & \text{otherwise}.
\end{cases}
\]
Here $\gde$ and $\gce$ are the graphs obtained from $G$ by either {\em deleting} or {\em contracting} 
$e$ respectively.  Our three parameters $\tau(G)$, $\alpha(G)$ and $\alpha^*(G)$ are all {\em evaluations} of the Tutte polynomial 
\begin{align*}
\tau(G) & = T_G(1,1),\\
\alpha(G) &= T_G(2,0),\\
\alpha^*(G) &= T_G(0,2), 
\end{align*}
and so they individually satisfy the ``deletion-contraction" identity. The fundamental problem preventing the Merino-Welsh conjecture from being a trivial exercise in induction is that {\em deletion} of an edge can create bridges, and 
{\em contraction} of an edge can create loops, thus creating graphs to which the inductive hypothesis does not apply.

Two edges $e$, $f$ are in {\em series} in a graph if neither is a bridge and any cycle containing $e$ also contains $f$, and are in {\em parallel} if neither is a loop and any edge-cutset containing $e$ also contains $f$. The {\em series class} of $e$ is the set $\sigma(e)$ of all edges in
series with $e$ and is denoted $\sigma(e)$, while the {\em parallel class} of $e$ is the set $\pi(e)$ of all edges in 
parallel with $e$. If $|\sigma(e)| > 1$ then deleting $e$ creates a bridge, while if $|\pi(e)| > 1$, then contracting $e$
creates a loop. However if $|\sigma(e)|=|\pi(e)|=1$, then $e$ can be safely deleted and contracted, and it is immediate that the additive Merino-Welsh conjecture (Conjecture~\ref{conj:CMadd}) holds for $G$ if it holds for both $G\backslash e$ and
$\gce$. To show that this is also true for the {\em multiplicative} Merino-Welsh conjecture (Conjecture~\ref{conj:CMmult}) we first need some lemmas due to Jackson \cite{MR2663569}. 

\begin{lemma}\label{lem:disc}
The condition $\alpha(G)\alpha^*(G)\geq \tau(G)^2$  
is equivalent to the statement that for all $\lambda\in\reals$,
\begin{equation}\label{eq:disc}
 \alpha(G)\lambda^2  - 2\tau(G)\lambda  + \alpha^*(G) \geq 0.\end{equation}\end{lemma}
\begin{proof}
The left-hand side of \eqref{eq:disc}, viewed as a polynomial in $\lambda$, is non-negative everywhere
if and only if its discriminant is non-positive. As the discriminant is 
\[
4\tau(G)^2-4\alpha(G)\alpha^*(G)
\]
the lemma follows.
\end{proof}

\begin{lemma}\label{lem:parser}
Suppose that $e$ is an edge of a graph $G$, neither a loop nor a bridge, such that $|\sigma(e)|=|\pi(e)|=1$. Furthermore, suppose that $\gde$ and $\gce$, both satisfy the multiplicative Merino-Welsh conjecture (Conjecture~\ref{conj:CMmult}). Then $G$ also satisfies the multiplicative Merino-Welsh conjecture.
\end{lemma}
\begin{proof}By the deletion-contraction identity, we have
\begin{align*}
\alpha(G)\lambda^2 - 2\tau(G) \lambda + \alpha^*(G) &= (\alpha(\gde)\lambda^2 - 2\tau(\gde)\lambda + \alpha^*(\gde))\\ &\phantom{=}{ } +
(\alpha(\gce)\lambda^2 - 2\tau(\gce)\lambda + \alpha^*(\gce)).\end{align*}
Neither $\gde$ nor $\gce$ has a loop or a bridge and so by Lemma~\ref{lem:disc}, both of the bracketed terms are positive. By applying Lemma~\ref{lem:disc} again, the result follows.
\end{proof}

Therefore in any minimal counterexample to the Merino-Welsh conjecture, every edge must lie in a non-trivial series or parallel class. On the other hand, none of the series or parallel classes can be very large.

\begin{lemma}\label{lem:par}
Suppose that $G$ is a graph that satisfies the multiplicative Merino-Welsh conjecture (Conjecture~\ref{conj:CMmult}), and that $e$ is an edge of $G$ with $|\pi(e)|\geq 2$. Let $G^+$ be formed from $G$ by adding two edges $f$ and $g$ in parallel with $e$. Then $G^+$ also satisfies the multiplicative Merino-Welsh conjecture.
\end{lemma}
\begin{proof}
Given a totally cyclic orientation of $G$, we may obtain a totally cyclic orientation of $G^+$ by orienting $f$ and $g$ in any way we choose. Consequently $\alpha^*(G^+) \geq 4\alpha^*(G)$.
On the other hand, $\alpha(G^+)=\alpha(G)$, because the addition of parallel edges does not alter the number of acyclic orientations. The number of spanning trees of $G^+$ which do not contain a member of $\pi_{G^+}(e)$ is the same as the number of spanning trees of $G$ which do not contain a member of $\pi_G(e)$. However the number of spanning trees of $G^+$ which do contain a member of $\pi_{G^+}(e)$ is equal to the number of spanning trees of $G$ which contain a member of $\pi_G(e)$ multiplied by $|\pi_{G^+}(e)|/|\pi_G(e)|$. Consequently 
\[
\tau(G^+) \leq |\pi_{G^+}(e)|/|\pi_G(e)| \tau(G) \leq 2\tau (G).\]
 Therefore $\alpha(G^+)\alpha^*(G^+) \geq 4\alpha(G)\alpha^*(G) \geq (2\tau(G))^2 \geq \tau(G^+)^2$.
\end{proof}

The following is the analogous version of the previous lemma with parallel replaced by series. The proof is similar but the roles of $\alpha$ and $\alpha^*$ are interchanged.
\begin{lemma}\label{lem:ser}
Suppose that $G$ is a graph that satisfies the multiplicative Merino-Welsh conjecture, and that $e$ is an edge of $G$ with $|\sigma(e)|\geq 2$. Let $G^+$ be formed from $G$ by adding two edges $f$ and $g$ in series with $e$. Then $G^+$ also satisfies the multiplicative Merino-Welsh conjecture. \qed
\end{lemma}

The consequence of these two results is that in any minimal counterexample to the multiplicative Merino-Welsh conjecture, every edge lies either in a series class of size 2 or 3, or a parallel class of size 2 or 3.

\section{Series--Parallel Graphs}


A \emph{two--terminal} graph $(G,s,t)$ is a graph $G$ with two distinct distinguished vertices, $s$ and $t$, called the {\em terminals}, with $s$ designated as the {\em source} and $t$ as the {\em sink}. A \emph{two--terminal series--parallel (TTSP)} graph is a two--terminal graph that can be constructed from the two--terminal graph $K_2$, with the sole edge connecting the source to the sink, by a sequence of the following operations:
\begin{enumerate}
\item (parallel connection) take two TTSP graphs $(G,s_G,t_G)$ and $(H,s_H,t_H)$, identify $s_G$ with $s_H$, forming the source of the new graph $G \pll H$, and identify $t_G$ with $t_H$, forming the sink of $G \pll H$;
\item (series connection) take two TTSP graphs $(G,s_G,t_G)$ and $(H,s_H,t_H)$, identify $t_G$ with $s_H$ and let $s_G$ and $t_H$ be the source and sink, respectively, of the new graph $G \ser H$.
\end{enumerate}
Both the series and parallel connection operations are associative.

Series-parallel graphs can be defined in various equivalent ways, for example as graphs with no $K_4$-minor,  or as graphs whose blocks are either single edges, or can be reduced to a loop by a sequence of operations each of which is the suppression of a vertex of degree two or the elimination of an edge in parallel to another edge. However for us, the key property of series-parallel graphs is their relationship to two--terminal series--parallel graphs as expressed in the following lemma:

\begin{lemma}\label{lem:2terminal}
A graph $G$ is a 2-connected series--parallel graph if and only if for every edge $e=st$, the two--terminal graph $(\gde, s, t)$ is a TTSP graph. \qed
\end{lemma}

Consider a rooted binary tree, with each non-leaf node designated as either an $s$-node or a $p$-node (see Figure~\ref{fig:sptree} for an example). We can use this tree as a ``blueprint'' to construct a TTSP graph by first associating the graph $K_2$ with each leaf and then, working up the tree, associating with each $s$-node the series connection of its 
two children, and with each $p$-node the parallel connection of its two children, and finally reading off the graph associated with the root of the tree. This tree is called the {\em decomposition} tree of the TTSP graph, and any TTSP graph can be
described by such a decomposition tree. 

\begin{figure}
\begin{center}
\begin{tikzpicture}[xscale=0.6]
\tikzstyle{vertex}=[inner sep=0.25mm, outer sep = 0.5mm]
\tikzstyle{leaf}=[circle,draw=black,fill=white, inner sep=0.75mm]
\node (v0) at (0,0) {\footnotesize $K_2$};
\node (v1) at (1,0) {\footnotesize $K_2$};
\node (v2) at (2,0) {\footnotesize $K_2$};
\node (v3) at (3,0) {\footnotesize $K_2$};
\node (v6) at (6,0) {\footnotesize $K_2$};
\node (v7) at (7,0) {\footnotesize $K_2$};
\node [vertex] (v8) at (0.5,1) { \small $p$};
\draw (v8)--(v0);
\draw (v8)--(v1);
\node [vertex] (v9) at (2.5,1) { $s$};
\draw (v9)--(v2);
\draw (v9)--(v3);
\node (v10) at (4.5,1) {\footnotesize $K_2$};
\node [vertex] (v11) at (6.5,1) {$s$};
\draw (v11)--(v6);
\draw (v11)--(v7);
\node [vertex] (v12) at (1.5,2) {$s$};
\node [vertex] (v13) at (5.5,2) {\small $p$};
\draw (v12)--(v8);
\draw (v12)--(v9);
\draw (v13)--(v10);
\draw (v13)--(v11);
\node [vertex] (v14) at (3.5,3) { \small $p$};
\draw (v14)--(v12);
\draw (v14)--(v13);
\end{tikzpicture}
\hspace{2cm}
\begin{tikzpicture}[out=45,in=135,relative]
\tikzstyle{vertex}=[circle, fill=gray,draw = black, inner sep = 0.6mm]
\node [vertex] (v0) at (0,0) {};
\node [vertex] (v1) at (1,-1) {};
\node [vertex] (v2) at (2,-1) {};
\node [vertex] (v3) at (3,0) {};
\node [vertex] (v4) at (1.5,1) {};
\draw [bend right] (v0) to (v1);
\draw [bend left] (v0) to (v1);
\draw (v1)--(v2);
\draw [bend right] (v2) to (v3);
\draw [bend left] (v0) to (v4);
\draw [bend left] (v4) to (v3);
\draw [bend left] (v0) to (v3);
\node (bottom) at (1,-1.5) {};
\end{tikzpicture}
\end{center}
\caption{A decomposition tree and associated TTSP graph}
\label{fig:sptree}
\end{figure}
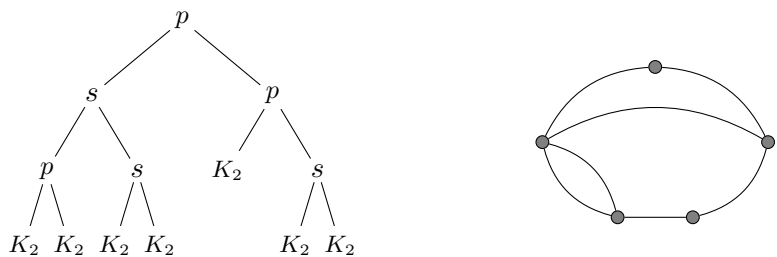

Any series-parallel graph is planar and its planar dual is again a series-parallel graph. However, we need a concept of duality---that we denote {\textit{sp}-duality}---that respects the terminals and allows us to remain within the class of two-terminal series--parallel graphs. Thus we say that two TTSP graphs $G$ and $H$ are \textit{sp}-dual if there are decomposition trees for $G$ and $H$ of identical shape, but with all the $s$-nodes changed to $p$-nodes, and vice versa.
There is however a close relationship between \textit{sp}-duality and planar duality: in particular, if $H = (H,s_H,t_H)$ is the \textit{sp}-dual of $G = (G,s_G,t_G)$ then there are plane embeddings of $G$ and $H$ such that $G+s_Gt_G$ is the planar dual of $H+s_Ht_H$.  We denote the \textit{sp}-dual of $G$ by $G^{*_\textit{sp}}$.

\begin{lemma}
If $G$, $H$ are TTSP graphs, then 
\begin{align*}
(G \ser H)^{*_\textit{sp}} &= G^{*_\textit{sp}} \pll H^{*_\textit{sp}},\\
(G \pll H)^{*_\textit{sp}} &= G^{*_\textit{sp}} \ser H^{*_\textit{sp}}.
\end{align*}
\end{lemma}
\begin{proof}
This follows immediately by considering the effect on the $sp$-tree of changing all $s$-nodes to $p$-nodes, and $p$-nodes to $s$-nodes.
\end{proof}

A \emph{2-forest} of a TTSP graph $G$ is a spanning subgraph of $G$ with two components, each of which is a tree, with one containing the sink of $G$ and the other containing the source of $G$. We denote the number of 2-forests of $G$ by $\tau_2(G)$. A \emph{very acyclic orientation} of a two-terminal series-parallel graph is an acyclic orientation in which there is no directed path between the two terminals. We denote the number of very acyclic orientations of $G$ by $\alpha_2(G)$. Finally, an \emph{almost totally cyclic orientation} of a two-terminal series-parallel graph is an orientation in which each edge either lies in a directed cycle or lies on a directed path between the two terminals. We denote the number of almost totally cyclic orientations of $G$ by $\alpha^*_2(G)$. The rationale behind introducing these additional parameters, thereby apparently complicating the problem, is that for a series-parallel graph, it is possible to keep track of these parameters through the operations of series and parallel connection.

\begin{lemma}
Suppose that $G = G_1 \ser G_2$. Then
\begin{align*} 
\tau(G) &= \tau(G_1)\tau(G_2),\\ 
\tau_2(G)&=\tau(G_1)\tau_2(G_2)+\tau_2(G_1)\tau(G_2),\\
\alpha(G)&=\alpha(G_1)\alpha(G_2),\\ \alpha_2(G)&= \alpha(G_1)\alpha(G_2)- \frac{(\alpha(G_1)-\alpha_2(G_1))(\alpha(G_2)-\alpha_2(G_2))}2,\\
\alpha^*_2(G)&= \alpha_2^*(G_1)\alpha_2^*(G_2)- \frac{(\alpha_2^*(G_1)-\alpha^*(G_1))(\alpha_2^*(G_2)-\alpha^*(G_2))}2,\\
\alpha^*(G)&=\alpha^*(G_1)\alpha^*(G_2).
\end{align*}
\end{lemma}

\begin{proof}
The arguments are straightforward counting arguments, and so we just give one example, namely the count for $\alpha_2(G)$. We first count the number of acyclic orientations of $G_1 \ser G_2$ that {\em do} admit a directed path between the terminals, noting that such any such path must either run source-to-sink or sink-to-source but {\em not both}. The restriction of such an acyclic orientation to $G_1$ must admit a directed path between its terminals, so there are $\alpha(G_1)-\alpha_2(G_1)$ such acyclic orientations and similarly for $G_2$. Exactly half of the resulting $(\alpha(G_1)-\alpha_2(G_1))(\alpha(G_2)-\alpha_2(G_2))$ combinations have the paths aligned consistently thereby creating a directed path between the terminals of $G_1 \ser G_2$. Subtracting this number from the total number of acyclic orientations of $G_1 \ser G_2$ gives the stated result. The other arguments are mild variants of this.
\end{proof}

\begin{lemma}
Suppose that $G = G_1 \pll G_2$. Then
\begin{align*}  
\tau(G)&=\tau(G_1)\tau_2(G_2)+\tau_2(G_1)\tau(G_2),\\ 
\tau_2(G) &= \tau_2(G_1)\tau_2(G_2),\\
\alpha(G)&= \alpha(G_1)\alpha(G_2)- \frac{(\alpha(G_1)-\alpha_2(G_1))(\alpha(G_2)-\alpha_2(G_2))}2,\\
\alpha_2(G)&=\alpha_2(G_1)\alpha_2(G_2),\\
\alpha^*_2(G)&=\alpha^*_2(G_1)\alpha^*_2(G_2),\\ 
\alpha^*(G)&= \alpha^*_2(G_1)\alpha^*_2(G_2)- \frac{(\alpha^*_2(G_1)-\alpha^*(G_1))(\alpha^*_2(G_2)-\alpha^*(G_2))}2.
\end{align*}
\end{lemma}

\begin{proof}
Straightforward counting arguments.
\end{proof}

Using the two previous lemmas, and some induction, we can immediately determine the relationship between the parameters of a TTSP graph
and its \textit{sp}-dual.
\begin{lemma}
If $H$ is the \textit{sp}-dual of $G$, then 
\[\tau(H)=\tau_2(G), \qquad \alpha(H)=\alpha^*_2(G), \qquad \alpha^*(H)=\alpha_2(G). \] \qed
\end{lemma}

\section{Replaceability and Reducibility}

A TTSP graph $G$ is \emph{replaceable} by a TTSP graph $H$ if for any TTSP graph $K$, $G\oplus_P K$ satisfies Conjecture~\ref{conj:CMmult} whenever $H\oplus_P K$ satisfies Conjecture~\ref{conj:CMmult}. We say that a TTSP graph is \emph{reducible} if it is replaceable by one with fewer edges. TTSP graphs which are not reducible are called \emph{irreducible}.

\begin{lemma}\label{lem:replace}
If the graph $G$ is replaceable by $H$, then for any $G'$, the TTSP graph $G \pll G'$ is replaceable by $H \pll G'$ and the TTSP graph $G \ser G'$ is replaceable
by $H \ser G'$.
\end{lemma}

\begin{proof}
First, we suppose that $(H \pll G') \pll K$ satisfies Conjecture~\ref{conj:CMmult}. Then as $(H \pll G') \pll K = H \pll (G' \pll K)$ and $G$ is replaceable by $H$, it follows that $G \pll (G' \pll K) = (G \pll G') \pll K$ also satisfies Conjecture~\ref{conj:CMmult}. Therefore $G \pll G'$ is replaceable by $H \pll G'$.

If $G = (G,s,t)$ is a 2-terminal graph, then temporarily let $r(G)$ denote the {\em reverse} 2-terminal graph $(G,t,s)$ where the roles of the source and sink have been reversed. 

Next we suppose that $(H \ser G') \pll K$ satisfies Conjecture~\ref{conj:CMmult}. Now $(H \ser G') \pll K$ is isomorphic
(as a graph, but not as a TTSP) to $H \pll (K \ser r(G'))$ and so $G \pll (K \ser r(G'))$ satisfies Conjecture~\ref{conj:CMmult}. But the latter is isomorphic to $(G \ser G') \pll K$ and so we have shown that $G \ser G'$ is replaceable
by $H \ser G'$. 
\end{proof}

We say that a TTSP graph $(G,s,t)$ is \emph{extendable} if either $G$ or $G+st$ has the property that every edge lies in a parallel class of size two or three, or lies in a series class of size two or three. If a TTSP graph $G$ is {\em not} extendable,
then nor is $G \ser H$ or $G \pll H$ for any graph $H$.

\begin{lemma}\label{lem:keybit}
If $G$ is an irreducible extendable TTSP graph with at least two edges, then it is either the series connection or the parallel connection of two smaller irreducible extendable TTSP graphs. 
\end{lemma}

\begin{proof}
As $G$ has at least two edges, then there are TTSP graphs $G_1$, $G_2$ such that either $G = G_1 \ser G_2$ or $G = G_1 \pll G_2$. It is clear that $G_1$ and $G_2$ are extendable, and by Lemma~\ref{lem:replace} they are irreducible.
\end{proof}

The proof that series-parallel graphs satisfy the Merino-Welsh conjectures hinges on showing that certain specific graphs are reducible, which requires demonstrating that it is replaceable. The next lemma gives a sufficient condition for this

\begin{lemma}\label{showreducible}
 Let $G$ be a TTSP graph. If there is a TTSP graph $H$
 \begin{multline}\label{minmax} \left( \max \{\tau(G)/\tau(H), \tau_2(G)/\tau_2(H)\} \right)^2\\
\leq \min\{\alpha(G)/\alpha(H),\alpha_2(G)/\alpha_2(H)\} \cdot \min \{\alpha^*(G)/\alpha^*(H),\alpha^*_2(G)/\alpha^*_2(H)\},\end{multline}
then $G$ is replaceable by $H$. (Here it is assumed that if $\alpha_2(H) = 0$ then the term $\alpha_2(G)/\alpha_2(H)$ is ignored, and similarly if $\alpha^*(H) = 0$.)
\end{lemma}

\begin{proof}
Recall that 
\begin{align*}
 \tau (H \pll K) &= \tau(H)\tau_2(K) + \tau_2(H)\tau(K),\\
 \alpha(H\pll K)&= \alpha(H)(\alpha(K)+\alpha_2(K))/2 + \alpha_2(H)(\alpha(K)-\alpha_2(K))/2,\\
 \alpha^*(H\pll K)&= \alpha^*(H)(\alpha^*_2(K)-\alpha^*(K))/2 + \alpha^*_2(H)(\alpha^*(K)+\alpha^*_2(K))/2.
\end{align*}
In other words, $\tau(H \pll K)$ is a linear combination of $\tau(H)$ and $\tau_2(H)$ whose coefficients depend on $K$ and $\tau(G \pll K)$ is the {\em same linear combination} with $\tau(G)$ and $\tau_2(G)$ replaced by $\tau(H)$ and $\tau_2(H)$; the same is true for the expressions for $\alpha$ and $\alpha^*$.

So suppose that $G$ and $H$ satisfy \eqref{minmax} and that $H \pll K$ satisfies the multiplicative Merino-Welsh conjecture. Then  
\begin{equation}\label{master}
( \tau(H) t_1 +  \tau_2(H) t_2 )^2 \leq ( \alpha(H) a_1+  \alpha_2(H)a_2) ( \alpha^*(H) c_1+  \alpha_2^*(H)c_2 )
\end{equation}
where $t_1$, $t_2$, $a_1$, $a_2$, $c_1$, $c_2$ are the coefficients depending on $K$. Replacing $H$ by $G$, the
three terms of this expression are changed as follows
\begin{align*}
(\tau(H) t_1 + \tau_2(H) t_2)^2 &\rightarrow (\tau(G) t_1 + \tau_2(G) t_2) ^2\\
(\alpha(H) a_1 + \alpha_2(H)a_2) &\rightarrow (\alpha(G) a_1 + \alpha_2(G) a_2) \\
(\alpha^*(H) c_1 + \alpha^*_2(H) c_2) &\rightarrow (\alpha^*(G) c_1 + \alpha_2^*(G) c_2) 
\end{align*}
Therefore, when $H$ is replaced by $G$, the left-hand side of \eqref{master} is multiplied by {\em at most} 
 \begin{equation}\label{lhs}
 \left( \max \{\tau(G)/\tau(H), \tau_2(G)/\tau_2(H)\} \right)^2
 \end{equation}
 while the right-hand side is multiplied by {\em at least}
 \begin{equation}\label{rhs}
 \min\{\alpha(G)/\alpha(H),\alpha_2(G)/\alpha_2(H)\} \cdot \min \{\alpha^*(G)/\alpha^*(H),\alpha^*_2(G)/\alpha^*_2(H)\},
 \end{equation}
 where any terms involving zero denominators are omitted.  By the hypotheses of the lemma, the expression \eqref{lhs} is  at most equal to the expression \eqref{rhs} and therefore
 $G \pll K$ satisfies the multiplicative Merino-Welsh conjecture. \end{proof}

It is straightforward to systematically construct all series-parallel graphs ordered by increasing number of edges---start with $K_2$, and at each stage form the series connection and parallel connection of one of the not-yet-processed pairs of graphs with the smallest total number of edges, adding the newly-constructed graphs to the growing list.  If a reducible graph $G$ is constructed during this process, then by Lemma~\ref{lem:replace}, any further graphs constructed using $G$ are also reducible. As previously noted, if a non-extendable graph is produced during this process, then any further graphs constructed using this are also non-extendable.  Therefore a modified procedure that immediately discards any graphs that are either non-extendable or that can be shown by Lemma~\ref{showreducible} to be reducible will produce a list of graphs that still contains all the extendable irreducible graphs (and perhaps some others). The surprise is that this modified procedure terminates, and indeed terminates quite quickly:

\begin{table}
\begin{tabular}{c|ccccccccc}
No.& Edges & \textit{sp}-dual & Built & $\tau$ & $\tau_2$ & $\alpha$ & $\alpha_2$ & $\alpha^*_2$ & $\alpha^*$\\
\hline
0 & 1 & 0 &  --- & 1 & 1 & 2 & 0 & 2 & 0 \\
1 & 2 & 2 & 0 $\oplus_P$ 0 & 2 & 1 & 2 & 0 & 4 & 2 \\
2 & 2 & 1 & 0 $\oplus_S$ 0 & 1 & 2 & 4 & 2 & 2 & 0 \\
3 & 3 & 4 & 0 $\oplus_P$ 2 & 3 & 2 & 6 & 0 & 4 & 2 \\
4 & 3 & 3 & 0 $\oplus_S$ 1 & 2 & 3 & 4 & 2 & 6 & 0 \\
5 & 3 & 6 & 0 $\oplus_P$ 1 & 3 & 1 & 2 & 0 & 8 & 6 \\
6 & 3 & 5 & 0 $\oplus_S$ 2 & 1 & 3 & 8 & 6 & 2 & 0 \\
7 & 4 & 8 & 1 $\oplus_S$ 1 & 4 & 4 & 4 & 2 & 14 & 4 \\
8 & 4 & 7 & 2 $\oplus_P$ 2 & 4 & 4 & 14 & 4 & 4 & 2 \\
9 & 4 & 10 & 1 $\oplus_P$ 2 & 5 & 2 & 6 & 0 & 8 & 6 \\
10 & 4 & 9 & 1 $\oplus_S$ 2 & 2 & 5 & 8 & 6 & 6 & 0 \\
11 & 5 & 12 & 1 $\oplus_S$ 5 & 6 & 5 & 4 & 2 & 30 & 12 \\
12 & 5 & 11 & 2 $\oplus_P$ 6 & 5 & 6 & 30 & 12 & 4 & 2 \\
13 & 6 & 14 & 2 $\oplus_P$ 8 & 12 & 8 & 46 & 8 & 8 & 6 \\
14 & 6 & 13 & 1 $\oplus_S$ 7 & 8 & 12 & 8 & 6 & 46 & 8 \\
15 & 7 & 16 & 5 $\oplus_S$ 8 & 12 & 16 & 28 & 18 & 30 & 12 \\
16 & 7 & 15 & 6 $\oplus_P$ 7 & 16 & 12 & 30 & 12 & 28 & 18 \\
17 & 7 & 18 & 2 $\oplus_P$ 12 & 16 & 12 & 102 & 24 & 8 & 6 \\
18 & 7 & 17 & 1 $\oplus_S$ 11 & 12 & 16 & 8 & 6 & 102 & 24\\
\hline
\end{tabular}
\caption{List containing all extendable irreducible graphs}
\label{tab:main}
\end{table}

\begin{table}
\begin{tabular}{c|ccccccccccccccccccc}
&0&1&2&3&4&5&6&7&8&9&10&11&12&13&14&15&16&17&18\\
\hline
0 & - & =4 & =6 & N&=10& 4  & 2  & 2  & 2  & 4  & 6  & 2  & 2  & 2  & 2  & 2  & 2  & 2  & 2  \\
1 & =5 & - & =10 & N&2& =11 & 6  & =14 & 4  & 7  & 2  & =18 & 2  & 7  & 2  & 2  & 7  & 7  & 2  \\
2 & =3 & =9 & - & N&6& 2  & 2  & 2  & 2  & 2  & 2  & 2  & 2  & 2  & 2  & 2  & 2  & 2  & 2  \\
3 & =9 & 5 & 1 &  -&N & N  & N  & N  & N  & N  & N  & N  & N  & N  & N  & N  & N  & N  & N\\
4 & N & N & N & N& - & 2  & 2  & 2  & 2  & 2  & 2  & 2  & 2  & 2  & 2  & 2  & 2  & 2  & 2 \\
5 & 1  & 1  & 5  & 1&N& - & 2  & =18 & =15 & 3  & 2  & 0  & 2  & 0  & 2  & 2  & 0  & 0  & 2  \\
6 & 3  & 1  & =12 & 1&N& 1  & - & 2  & 2  & 2  & 2  & 2  & 2  & 2  & 2  & 2  & 2  & 2  & 2  \\
7 & 1  & 1  & 3  & 1&N& 1  & =16 & - & 2  & 4  & 2  & 2  & 2  & 2  & 2  & 2  & 2  & 2  & 2  \\
8 & 1  & 1  & =13 & 1&N& 1  & =17 & 1  & - & 4  & 2  & 2  & 2  & 2  & 2  & 2  & 2  & 2  & 2  \\
9 & 5  & 1  & 1  & 1&N& 1  & 1  & 1  & 1  & - & 2  & 7  & 2  & 0  & 2  & 2  & 0  & 0  & 2  \\
10 & 3  & 1  & 8  & 1&N& 1  & 4  & 3  & 3  & 1  & - & 2  & 2  & 2  & 2  & 2  & 2  & 2  & 2  \\
11 & 1  & 1  & 1  & 1&N& 1  & 1  & 1  & 1  & 1  & 1  & - & 2  & 4  & 2  & 2  & 2  & 2  & 2  \\
12 & 1  & 1  & =17 & 1&N& 1  & 0  & 1  & 1  & 1  & 8  & 1  & - & 2  & 2  & 2  & 2  & 2  & 2  \\
13 & 1  & 1  & 1  & 1&N& 1  & 1  & 1  & 1  & 1  & 1  & 1  & 1  & - & 2  & 2  & 2  & 4  & 2  \\
14 & 1  & 1  & 8  & 1&N& 1  & 0  & 1  & 1  & 1  & 0  & 1  & 3  & 1  & - & 2  & 2  & 2  & 2  \\
15 & 1  & 1  & 8  & 1&N& 1  & 0  & 1  & 1  & 1  & 0  & 1  & 1  & 1  & 1  & - & 2  & 2  & 2  \\
16 & 1  & 1  & 1  & 1&N& 1  & 1  & 1  & 1  & 1  & 1  & 1  & 1  & 1  & 1  & 1  & - & 0  & 2  \\
17 & 1  & 1  & 1  & 1&N& 1  & 1  & 1  & 1  & 1  & 1  & 1  & 1  & 1  & 1  & 1  & 1  & - & 2  \\
18 & 1  & 1  & 8  & 1&N& 1  & 0  & 1  & 1  & 1  & 0  & 1  & 1  & 1  & 3  & 0  & 1  & 1  & -\\
\hline
\end{tabular}
\caption{Each pair $G_1\ser G_2$ and $G_1 \pll G_2$ is either in Table~\ref{tab:main}, is non-extendable or is reducible}
\label{tab:combs}
\end{table}

\begin{table}
\begin{tabular}{c|ccccccccccccccccccccc}
$G$&0&1&2&3&4&5&6&7&8&9&10&11&12&13&14&15&16&17&18\\
\hline
$G\oplus_S G$ & =2  & =7  & 2  & N&2& 3  & 2  & 2  & 2  & 7  & 2  & 2  & 2  & 4  & 2  & 2  & 0  & 0  & 2  \\
$G\oplus_P G$ & =1  & 1  & =8  & 1&N& 1  & 4  & 1  & 1  & 1  & 7  & 1  & 1  & 1  & 3  & 0  & 1  & 1  & 0 \\
\hline
\end{tabular}
\caption{Each pair $G\ser G$ and $G \pll G$ is either in Table~\ref{tab:main}, is non-extendable or is reducible}
\label{tab:combs2}
\end{table}

\begin{proposition}
Every extendable TTSP graph is either listed in Table~\ref{tab:main} or is reducible.
\end{proposition}

\begin{proof}
Suppose for a contradiction that there is an extendable irreducible TTSP graph not listed in Table~\ref{tab:main}, and let $G$ be such a graph with fewest edges.  As $K_2$ is in the table, $G$ has more than one edge and so $G = G_1 \ser G_2$
or $G = G_1 \pll G_2$, where both $G_1$ and $G_2$ are extendable and irreducible. By the minimality of $G$, both 
$G_1$ and $G_2$ occur in Table~\ref{tab:main}. However it is easy to check for each pair of graphs in Table~\ref{tab:main} that both $G_1 \ser G_2$ and $G_1 \pll G_2$ are either included in Table~\ref{tab:main}, are not extendable, or can be shown to be reducible by Lemma~\ref{showreducible}.  This information is summarised in Table~\ref{tab:combs} in the following manner:  the graphs in Table~\ref{tab:main} are numbered $0$, $1$, $\ldots$, $18$ and the rows and columns of
Table~\ref{tab:combs} are indexed by these graphs. The above-diagonal entries of Table~\ref{tab:combs} give information about the  to the series connection of the corresponding graphs, while the below-diagonal entries refer to the parallel connection. Each entry is either $N$, indicating that the graph constructed is not extendable, or is a plain integer $x$ indicating that the graph constructed is reducible because it can be replaced by the smaller graph $x$ (using Lemma~\ref{showreducible}), or an integer preceded by an equals sign  =$x$ indicating that the graph is isomorphic to graph $x$, and hence in Table~\ref{tab:main}.  Table~\ref{tab:combs2} gives the same information for the series connection and parallel connection when the two components
are isomorphic.
\end{proof}

\begin{theorem}
If $G$ is a series--parallel graph without loops or bridges, then
\begin{equation}\label{eq:main} \alpha(G)\alpha^*(G) \geq \tau(G)^2.\end{equation}
\end{theorem}

\begin{proof}
Suppose for a contradiction that $G$ is a series--parallel graph not satisfying Equation~\eqref{eq:main}, and that among all such graphs $G$ has the fewest edges. 
Then $G$ is 2-connected, because each of $\alpha$, $\alpha^*$ and $\tau$ is multiplicative over blocks.
Let $e$ be an edge of $G$ with endvertices $u$ and $v$. Then, by Lemma~\ref{lem:2terminal}, $(\gde, u, v)$ is a two-terminal series-parallel graph. By Lemmas~\ref{lem:par} and \ref{lem:ser}, each edge of $G$ lies in a series class of size two or three, or in a parallel class of size two or three. Consequently every edge of $\gde$ lies in a series class of size two or three, except possibly an edge joining $u$ and $v$ or an edge which is contained in every path from $u$ to $v$ in $\gde$, and therefore $\gde$ is extendable. Since $G$ is a minimal counterexample to the theorem, it follows from \ref{lem:replace} that $\gde$ is  irreducible. Therefore $\gde$ is one of the graphs listed in Table~\ref{tab:main}. However, for each of these graphs, it is easy to check that if an edge is added between the terminals, then the resulting graph satisfies Equation~\eqref{eq:main}, thereby supplying the required contradiction.
\end{proof}

\section{Conclusion}

All three of the parameters $\tau(G)$, $\alpha(G)$ and $\alpha^*(G)$ are evaluations of the Tutte polynomial $T_G(x,y)$; we have 
\begin{align*}
\tau(G) & = T_G(1,1),\\
\alpha(G) &= T_G(2,0),\\
\alpha^*(G) &= T_G(0,2).
\end{align*}
As the Tutte polynomial is naturally defined for all {\em matroids}, the conjecture can directly be extended to any matroid $M$ where it becomes 
\[
T_M(1,1) \leq \max\{T_M(2,0), T_M(0,2)\}
\]
even though there are no obvious combinatorial interpretations of $T_M(2,0)$ or $T_M(0,2)$ for general matroids. 

A matroid is called a {\em paving matroid} if it's smallest circuit has size at least equal to its rank, and it is generally presumed (though not proved) that asymptotically almost all matroids are paving matroids. Ch{\'a}vez-Lomel{\'{\i}}, Merino, Noble and Ram{\'{\i}}rez-Ib{\'a}{\~n}ez \cite{MR2764805} proved (among other results) that for a coloopless paving matroid $M$, the function $T_M(1-x,1+x)$ is {\em convex} in the region $-1 \leq x \leq 1$, thus proving that paving matroids satisfy the Merino-Welsh conjecture.

So the conjecture is proved for the vast class of paving matroids, and now for the tiny class of series-parallel graphs, but for general graphs and matroids, every variant of the conjecture remains open. While it may be possible to develop bounds for dense or sparse matroids similar to those found by Thomassen for graphs, the heart of the problem lies in the case where the rank is half the number of elements. 

\appendix

\section{Code}

This appendix contains the Mathematica code used to check the assertions made in Table~\ref{tab:combs} and Table~\ref{tab:combs2} regarding the replaceability of the series and parallel connections of each of the graphs in Table~\ref{tab:main}.  In this code, the parameters for each graph are represented as a vector 
\[
\left( \tau(G), \tau_2(G), \alpha(G), \alpha_2(G), \alpha_2^*(G), \alpha^*(G) \right)
\]
and the functions \verb+ser+ and \verb+par+ produce the parameters of the series connection and parallel connection respectively.

\begin{verbatim}


ser::usage="Returns the parameters for the series connection of g and h"

ser[g_,h_] := {
  g[[1]]h[[1]], 
  g[[1]]h[[2]]+g[[2]]h[[1]],
  g[[3]]h[[3]],
  g[[3]]h[[3]]-(g[[3]]-g[[4]])(h[[3]]-h[[4]])/2,
  g[[5]]h[[5]]-(g[[5]]-g[[6]])(h[[5]]-h[[6]])/2,
  g[[6]]h[[6]]};

par::usage="Returns the parameters for the parallel connection of g and h"

par[g_,h_] := {
  g[[1]]h[[2]]+g[[2]]h[[1]],
  g[[2]]h[[2]],
  g[[3]]h[[3]]-(g[[3]]-g[[4]])(h[[3]]-h[[4]])/2,
  g[[4]]h[[4]],
  g[[5]]h[[5]],
  g[[5]]h[[5]]-(g[[5]]-g[[6]])(h[[5]]-h[[6]])/2};

spdual::usage="Returns the parameters for the sp-dual of g"

spdual[g_] := {
 g[[2]],g[[1]],g[[5]],g[[6]],g[[3]],g[[4]]};
\end{verbatim}

\noindent
The function \verb+replaces+ is a boolean function testing whether the first argument can be
replaced by the second.

\begin{verbatim}
replaces::usage="Returns true if g can be replaced by h, else false"

replaces[g_,h_] := Module[{t1,t2,t3}, 
  t1 = Max[ g[[1]]/h[[1]], g[[2]]/h[[2]] ];
  t2 = If[h[[4]]==0, g[[3]]/h[[3]], Min [ g[[3]]/h[[3]], g[[4]]/h[[4]] ]];
  t3 = If[h[[6]]==0, g[[5]]/h[[5]], Min [ g[[5]]/h[[5]], g[[6]]/h[[6]] ]];
  t1^2-t2 t3 <= 0]
\end{verbatim}

\noindent
Finally, \verb+gs+ is the collection of all the graphs (other than $K_2$) in Table~\ref{tab:main} ordered such that \verb+gs[[x]]+ is the graph \verb+x+ in the table.
 
\begin{verbatim}
k2 = {1,1,2,0,2,0}; 
g0 = k2;
g1 = par[g0,g0]; g2 = ser[g0,g0]; 
g3 = par[g0,g2]; g4 = ser[g0,g1];
g5 = par[g0,g1]; g6 = ser[g0,g2];
g7 = ser[g1,g1]; g8 = par[g2,g2];
g9 = par[g1,g2]; g10 = ser[g1,g2];
g11 = ser[g1,g5]; g12 = par[g2,g6];
g13 = par[g2,g8]; g14 = ser[g1,g7];
g15 = ser[g5,g8]; g16 = par[g6,g7];
g17 = par[g2,g12]; g18 = ser[g1,g11];

gs = {g1,g2,g3,g4,g5,g6,g7,g8,g9,g10,g11,g12,g13,g14,g15,g16,g17,g18}
\end{verbatim}

It is now straightforward to confirm any of the assertions contained in the tables; for example, one entry in the table claims that $G_1 \ser G_6$ is replaceable by $G_6$ itself. This can easily be confirmed by ensuring that the  statement
\begin{verbatim}
replaces[ ser[ gs[[1]], gs[[6]] ], gs[[6]] ]
\end{verbatim}
returns \verb+True+, and similarly for all the other entries.

\bibliographystyle{acm2url}
\bibliography{/Users/gordon/Dropbox/gordonmaster}

\begin{thebibliography}{1}

\bibitem{MR2764805}
{ Ch{\'a}vez-Lomel{\'{\i}}, L.~E., Merino, C., Noble, S.~D., and
  Ram{\'{\i}}rez-Ib{\'a}{\~n}ez, M.}
\newblock Some inequalities for the {T}utte polynomial.
\newblock { European J. Combin. 32}, 3 (2011), 422--433.
\newblock Available from: \url{http://dx.doi.org/10.1016/j.ejc.2010.11.005}.

\bibitem{MR2555382}
{ Conde, R., and Merino, C.}
\newblock Comparing the number of acyclic and totally cyclic orientations with
  that of spanning trees of a graph.
\newblock { Int. J. Math. Comb. 2\/} (2009), 79--89.

\bibitem{MR2663569}
{ Jackson, B.}
\newblock An inequality for {T}utte polynomials.
\newblock { Combinatorica 30}, 1 (2010), 69--81.
\newblock Available from: \url{http://dx.doi.org/10.1007/s00493-010-2484-4}.

\bibitem{MR1772357}
{ Merino, C., and Welsh, D. J.~A.}
\newblock Forests, colorings and acyclic orientations of the square lattice.
\newblock { Ann. Comb. 3}, 2-4 (1999), 417--429.
\newblock On combinatorics and statistical mechanics.
\newblock Available from: \url{http://dx.doi.org/10.1007/BF01608795}.

\bibitem{MR2732508}
{ Thomassen, C.}
\newblock Spanning trees and orientations of graphs.
\newblock { J. Comb. 1}, 2 (2010), 101--111.

\end{thebibliography}

\end{document}